\documentclass[a4paper, 10pt,leqno]{amsart}
\date{\today}
\usepackage[dvipdfmx]{graphicx}
\usepackage{latexsym}
\usepackage{amssymb}
\usepackage{amsmath}
\usepackage{amscd}
\usepackage{amsthm}
\usepackage{amsfonts}
\usepackage{mathrsfs}
\usepackage{caption}
\usepackage{enumerate}
\usepackage{enumitem}
\usepackage{comment}
\usepackage{bm}
\usepackage[usenames]{color}
\usepackage[active]{srcltx}
\usepackage{comment}
\usepackage[bbgreekl]{mathbbol}
\usepackage{here}
\usepackage[subrefformat=parens]{subcaption}
  \DeclareSymbolFontAlphabet{\mathbb}{AMSb}
  \DeclareSymbolFontAlphabet{\mathbbl}{bbold}
  \DeclareMathSymbol{\bbepsilon}{\mathord}{bbold}{"0F}
\title{Minimal total absolute curvature for equiaffine immersions}

\author[Y.~Yamauchi]{Yuta Yamauchi}
\address{%
   Graduate School of Engineering Science, 
   Yokohama National University, 
   Hodogaya, Yokohama 240-8501, Japan
}
\email{yamauchi-yuta-hj@ynu.jp}


\usepackage{amsthm}
\theoremstyle{plain}
 \newtheorem{theorem}{Theorem}[section]
 \newtheorem{introtheorem}{Theorem}

 \newtheorem{proposition}[theorem]{Proposition}
 \newtheorem{fact}[theorem]{Fact}
 \newtheorem*{fact*}{Fact}
 \newtheorem{lemma}[theorem]{Lemma}
 
 \theoremstyle{remark}
 \newtheorem{definition}[theorem]{Definition}
 \newtheorem{remark}[theorem]{Remark}
 \newtheorem*{acknowledgements}{Acknowledgements}
 
 \newtheorem{example}[theorem]{Example}
\numberwithin{equation}{section}

 \makeatletter
    
    \@addtoreset{equation}{section}
  \makeatother

\newcommand{\bmath}[1]{\boldsymbol {#1}}

    \newcommand{\vect}[1]{\bmath{#1}}
\newcommand{\R}{{\boldsymbol R}}
\newcommand{\E}{{\boldsymbol E}}

\begin{document}

\begin{abstract}    
Koike (2001) defined the Lipschitz--Killing curvature and established a Chern--Lashof type inequality 
for equiaffine immersions of arbitrary codimensions. 
In this paper, we study the equality case.  
We prove that the total absolute curvature of an $n$-dimensional equiaffine immersion 
is equal to $2$ if and only if the image is a convex hypersurface embedded 
in an $(n+1)$-dimensional affine subspace.  

\end{abstract}

\maketitle


\section{Introduction}

Affine differential geometry, founded by 
Blaschke \cite{Bl}, 
is a natural generalization 
of Euclidean geometry.
A rich body of research 
has been accumulated up to the present day
(see \cite{Cal,Li,Mar,Matsu} for example). 
Recently, affine geometric viewpoints have attracted renewed attention
through interactions with information geometry, such as dual connections and
statistical manifolds \cite{Ama, AN}.


For equiaffine immersions, 
Koike \cite{Koi1} introduced
the Lipschitz--Killing curvatures and the total absolute curvature (see Definition~\ref{def:LK-cur} and Definition~\ref{def:TAC}).
Roughly speaking, the total absolute curvature measures
the global bending of the immersion in the ambient space.
Koike proved that a Chern--Lashof type theorem holds in the equiaffine
setting.

\begin{fact}[\cite{Koi1}]\label{fa:Koi}
Let $(M,\nabla,\theta)$ be an oriented compact $n$-dimensional manifold $M$ 
with an equiaffine structure $(\nabla, \theta)$, and let
$f:(M,\nabla,\theta) \to (\R^{n+r},\tilde{\nabla}, \omega)$ 
be an equiaffine immersion into an $(n+r)$-dimensional affine space $\R^{n+r}$
with the natural equiaffine structure $(\tilde{\nabla}, \omega)$.
\begin{itemize}
\item[$(1)$]
The total absolute curvature $\tau_S(f)$ of $f$ with respect to a unit
ellipsoid $S$ satisfies
\[
\tau_S(f) \geq \sum_{k=0}^n b_k(M,\boldsymbol{F}),
\]
where $b_k(M,\boldsymbol{F})$ denotes the $k$-th Betti number of $M$ with
respect to an arbitrary coefficient field $\boldsymbol{F}$.
\item[$(2)$]
If $\tau_S(f) < 3$, then $M$ is homeomorphic to the $n$-dimensional sphere.
\end{itemize}
\end{fact}

Chern and Lashof proved in \cite{CL1,CL2} that the total absolute curvature of an
immersion of a compact manifold into Euclidean space is at least the sum of the Betti numbers and
is equal to $2$ if and only if the image is a convex hypersurface embedded in an
$(n+1)$-dimensional affine subspace.
For equiaffine immersions, Koike defined the total absolute curvature $\tau_S(f)$ with respect to a unit ellipsoid $S$. 
Unlike the Euclidean case, this quantity generally depends on the choice of $S$ (cf. \eqref{eq:ellipsoid}). 
Nevertheless, a remarkable observation due to Koike is that this dependence disappears 
when $\tau_S (f)$ is the minimal number of Morse critical points on $M$ (Fact~\ref{fa:min}).
In particular, if $\tau_S(f)=2$ for some unit ellipsoid $S$, then $\tau_S(f)=2$ for every unit ellipsoid. 
Thus, although $\tau_S(f)$ is not an intrinsic quantity in general, 
the condition $\tau_S(f)=2$ is intrinsic to the equiaffine immersion.

What remains unclear in Koike's work is the geometric meaning of this minimal value.
This naturally leads to the following question:
does the condition $\tau_S(f)=2$ characterize convex hypersurfaces also in the equiaffine setting?
The main result of this paper gives an answer to this question.
More precisely, we prove the following theorem.

\begin{introtheorem}\label{thm:main}

Let $(M,\nabla,\theta)$ be an oriented compact $n$-dimensional manifold $M$ 
with an equiaffine structure $(\nabla, \theta)$, and let
$f:(M,\nabla,\theta) \to (\R^{n+r},\tilde{\nabla}, \omega)$ 
be an equiaffine immersion into an $(n+r)$-dimensional affine space $\R^{n+r}$
with the natural equiaffine structure $(\tilde{\nabla}, \omega)$.
If the total absolute curvature $\tau_S (f)$ of $f$
is equal to $2$, then the image $f(M)$
is a convex hypersurface embedded in an $(n+1)$-dimensional affine subspace of $\R^{n+r}$.
The converse is also true.
\end{introtheorem}

To establish this characterization, two difficulties must be overcome.
The first problem is the \emph{reduction of codimension}. 
In the Euclidean case, a natural metric is induced on a subspace, 
and therefore the total absolute curvature is preserved when the codimension is reduced.
By contrast, in the affine case, 
a subspace does not generally inherit a canonical equiaffine structure 
and hence it is not clear whether the value of the total absolute curvature is preserved. 
The key point of this paper is that, 
even though an affine subspace does not inherit a canonical equiaffine structure, 
we construct an appropriate structure that preserves the total absolute curvature 
under reduction of codimension in Proposition~\ref{prop:equiaffine_restrict}.
The main technical device of this paper is Proposition~\ref{prop:equiaffine_restrict},
which shows that one can give a suitable affine subspace with an appropriate equiaffine structure
so that the restricted immersion remains equiaffine and still has total absolute curvature $2$.
This enables the reduction of codimension while preserving the minimality of the total absolute curvature.

The second problem is that
\emph{the proof of Chern and Lashof essentially relies on a metric structure}.
Because of this, their argument cannot be directly applied to affine spaces, which do not possess a metric. 
To overcome this, we develop a new proof that does not rely on any metric, 
using only the properties of height functions and the Gauss map.
As a byproduct, this approach also yields an alternative proof in the Euclidean case.

Our theorem shows that \emph{locally strictly convex hypersurfaces}  
---often assumed a priori in affine differential geometry---
can be characterized by a global curvature condition in arbitrary codimensions.
Furthermore, the minimal total absolute curvature completely characterizes convex hypersurfaces, 
even when they are not locally strictly convex.
In addition, our study does not require non-degeneracy of the affine fundamental form,
which motivates us to consider equiaffine immersions with degenerate affine
fundamental form.
Even for convex hypersurfaces,
the affine fundamental form may degenerate.
Nevertheless, such hypersurfaces can still carry
degenerate affine metrics with meaningful geometric properties
(Example~\ref{ex:semidefinite}).
This suggests that the framework of equiaffine geometry can be meaningfully
extended beyond the non-degenerate setting.
In Section~\ref{sec:example}, we present an example of a hypersurface equipped
with a degenerate affine metric of Kossowski type (cf.\ \cite{HHNSUY}).

This paper is organized as follows.
In Section~\ref{sec:pre}, we review basic notions of equiaffine immersions and
the total absolute curvature.
In Section~\ref{sec:proof}, we prove
Theorem~\ref{thm:main}.
Finally, in Section~\ref{sec:example}, we present an example of
a convex hypersurface and a degenerate affine metric example in affine space.

\section{Preliminaries}\label{sec:pre}

In this section, we recall the notion of equiaffine immersions of arbitrary codimension
and the associated Lipschitz--Killing curvature, following \cite{Koi1,Koi2}.
Some properties of equiaffine immersions of arbitrary codimension are also described
in Notes $1$ and $2$ of \cite{NS}.
Throughout this paper, we assume that every affine connection is torsion-free.
Let $f : M \to (\R^{n+r}, \tilde{\nabla}, \omega)$ be an affine immersion of an
$n$-dimensional manifold into the affine space $\R^{n+r}$ with an equiaffine
structure $(\tilde{\nabla}, \omega)$.  
Let $\pi:N \to M$ be a transversal bundle of $f$ and $\theta^{\perp}$ be a 
volume element on $N$.
Then an affine connection $\nabla$ on $M$ and 
symmetric $N$-valued $(0,2)$-tensor field $\alpha$ are defined by
\[
\widetilde{\nabla}_X df(Y) =  df(\nabla_X Y) + \alpha(X,Y)
\]
where $X,Y$ are tangent vector fields on $M$.
The term $df(\nabla_X Y)$ is the $df(TM)$-component
 and $\alpha(X,Y)$ is the $N$-component of $\widetilde{\nabla}_X df(Y)$.
We call $\nabla$ the \emph{induced connection} by $N$ 
and $\alpha$ the \emph{affine fundamental form} of $(f,N)$.  
A volume element $\theta$ on $M$ is defined by 
\[
    \theta (X_1, \cdots , X_n) =
    \frac{\omega (df(X_1), \cdots , df(X_n), \xi_1 , \cdots , \xi_r )}
          {\theta^{\perp} (\xi_1 , \cdots , \xi_r)} 
    \quad (X_1, \cdots ,X_n \in TM)
\]
where $(\xi_1 , \cdots , \xi_r)$ is a frame of $N$.
We call $\theta$ the \emph{induced volume element} by $(N,\theta^\perp)$.
Let $(\nabla, \theta)$ be an equiaffine structure of $M$, i.e. $\nabla \theta = 0$.
If the immersion $f:(M,\nabla, \theta) \to (\R^{n+r}, \tilde{\nabla}, \omega) $ 
admits a pair $(N,\theta^\perp)$ inducing $(\nabla, \theta)$,
then the immersion $f$ is called an \emph{equiaffine immersion}.
In the hypersurface case $r=1$, 
a symmetric $(0,2)$-tensor field $\alpha_\xi$ is defined by 
\begin{equation}\label{eq:affine-fund}
   \alpha_\xi(X,Y) \xi =  \widetilde{\nabla}_X df(Y) -  df(\nabla_X Y) \quad (X, Y \in TM)
\end{equation}
where $\xi$ is a transversal vector field such that $\theta^\perp (\xi) =1$.
This $\alpha_\xi$ is called the \emph{affine fundamental form} of $(f,\xi)$.

Since an affine space does not carry a canonical inner product,
we cannot define the height function by the inner product as in the Euclidean case.
Hence, we introduce the height function associated with an element of $\bigwedge^{n+r-1}\R^{n+r}$
For each $\phi \in \bigwedge^{n+r-1}\R^{n+r}$, write
$\phi = \vect{v}_1\wedge \cdots \wedge \vect{v}_{n+r-1}$ and take $\vect{v}_{n+r}\in \R^{n+r}$ such that
$\omega(\vect{v}_1, \cdots, \vect{v}_{n+r-1}, \vect{v}_{n+r})=1$.  
Then we define the \emph{height function} $\tilde{h}_{\phi}$ \emph{associated with} $\phi$ by
\begin{equation}\label{eq:height-function}
       {\rm pr}_{\vect{v}_{n+r}} (v) = \tilde{h}_{\phi} (v) \cdot \vect{v}_{n+r} \quad (v \in \R^{n+r}),   
\end{equation}
where ${\rm pr}_{\vect{v}_{n+r}}$ is the projection onto $\operatorname{span} \{\vect{v}_{n+r}\}$ 
with respect to the decomposition 
$\operatorname{span} \{ \vect{v}_1, \cdots, \vect{v}_{n+r-1} \} \oplus \operatorname{span} \{\vect{v}_{n+r}\} = \R^{n+r}$, 
and $\cdot$ denotes the scalar product.
It is easy to show that this definition is independent of the representation of $\phi$.
The function $\tilde{h}_{\phi}$ can be interpreted as 
the affine analogue of a linear height function determined 
by the hyperplane $\operatorname{span} \{ \vect{v}_1, \cdots, \vect{v}_{n+r-1} \}$.

Because the unit sphere cannot be defined in the affine space,
the role of the unit sphere is replaced by a unit ellipsoid.
For the volume element $\omega$ of $\R^{n+r}$,
we define a volume element $\omega'$ of $\wedge^{n+r-1}\R^{n+r}$ by
\[
\omega' (\widehat{\vect{v}_1} \wedge \cdots \wedge \vect{v}_{n+r}, \cdots , \vect{v}_1\wedge \cdots \wedge \widehat{\vect{v}_{n+r}}) = (-1)^{(n+r)(n+r-1)/2}
\]
for a basis $(\vect{v}_1, \cdots, \vect{v}_{n+r})$ of $\R^{n+r}$ satisfying $\omega (\vect{v}_1, \cdots , \vect{v}_{n+r}) = 1$.
Here $\widehat{\ }$ denotes omission. 
Take a basis $(\zeta_1, \cdots , \zeta_{n+r})$ of $\wedge^{n+r-1}\R^{n+r}$ satisfying
$
\omega'(\zeta_1, \cdots , \zeta_{n+r}) = 1,
$
and define a set $S$ by
\begin{equation}\label{eq:ellipsoid}
S := \left\{a_1 \zeta_1 + \cdots a_{n+r} \zeta_{n+r} \in \wedge^{n+r-1}\R^{n+r} \;\middle|\; \sum_{i=1}^{n+r} (a_i)^2 =1\right\}.
\end{equation}
Such a $S$ is called the \emph{unit ellipsoid}.
Since $S$ depends on the choice of basis $(\zeta_1, \cdots , \zeta_{n+r})$,
 we fix one $S$ throughout this paper.

From now on, we fix a pair $(N,\theta^\perp)$.
We define a map $\tilde{\nu}_N:\bigwedge^{r-1}N \to \bigwedge^{n+r-1}\R^{n+r}$ by
\[
    \tilde{\nu}_N(\eta) := df(e_1)\wedge\cdots\wedge df(e_n) \wedge \eta
\]
where $(\vect{e}_1,\cdots,\vect{e}_n)$ is a frame of $TM$ with $\theta(\vect{e}_1,\cdots,\vect{e}_n)=1$. 
The \emph{transversal ellipsoid bundle} $B$ of $f$ is defined by
\[
    B := \tilde{\nu}_N^{-1}(S)\subset \bigwedge^{r-1}N.
\]
The restriction $\nu:=\tilde{\nu}_N|_B : B\to S$ is called the \emph{Gauss map} of $(f,N,\theta^{\perp},S)$.

\begin{definition}[\cite{Koi1}]\label{def:LK-cur}
For $\eta\in B$, we define a function $G: B \to \R$ by
\[
    G(\eta):=(-1)^n \det_\theta(\tilde{h}_\eta\circ \alpha),
\]
where $\det_\theta(\tilde{h}_\eta\circ \alpha)$ is the determinant of a matrix $[(\tilde{h}_\eta\circ \alpha)(\vect{e}_i,\vect{e}_j)]_{1 \leq i,j \leq n}$ 
for a frame $(\vect{e}_1,\cdots,\vect{e}_n)$ of $T_{\pi(\eta)} M$ with $\theta (\vect{e}_1,\cdots,\vect{e}_n) =1$.
We call $G$ the \emph{Lipschitz--Killing curvature} of $(f,N,\theta^{\perp},S)$.
\end{definition}

Let $\omega_B$ and $\omega_S$ be a volume element on $B$ and $S$, respectively (see \cite{Koi1}).
Then the following relation between $\omega_B$ and $\omega_S$ holds \cite[Proposition 2.2]{Koi1}:
\begin{equation}\label{eq:pullback}
    (\nu^*\omega_S)_\eta = G(\eta) \cdot (\omega_B)_\eta \quad (\eta \in B)
\end{equation}
where $\nu^*$ denotes the pullback by $\nu$.

Koike defined the total absolute curvature of equiaffine immersions as follows.

\begin{definition}\label{def:TAC}
    Let $(M,\nabla,\theta)$ be an oriented compact $n$-dimensional manifold $M$ 
    with an equiaffine structure $(\nabla, \theta)$.
    The \emph{total absolute curvature} $\tau_S (f)$ 
    of an equiaffine immersion $f:(M,\nabla,\theta) \to (\R^{n+r}, \tilde{\nabla}, \omega)$ 
    with respect to $S$ is defined by 
    \[
    \tau_S (f) := \frac{1}{\operatorname{vol}(\mathcal{S}^{n+r-1}(1))} \int_B |G|\, \omega_B ,
    \]
    where $\operatorname{vol}(\mathcal{S}^{n+r-1}(1))$ is the volume of the unit sphere $\mathcal{S}^{n+r-1}(1)$.
\end{definition}

Let $C \subset B$ denote the set of critical points of the Gauss map $\nu : B \to S$.
By Sard's theorem, the image $\nu(C)$ has measure zero.
Moreover, a height function $h_{\phi} = \tilde{h}_{\phi} \circ f:M \to R$ of $f$
is \emph{not} a Morse function if and only if $\phi \in \nu(C)$. 
Lemma 3.1 in \cite{Koi1} yields a point $p\in M$ is a critical point of $h_\phi$ if and only if 
there uniquely exists $\eta \in B_p$ satisfying $\nu(\eta) = \phi$.
Hence, the number of critical points of the Morse height function $h_\phi$ coincide with the number of the preimage $\nu^{-1} (\phi)$.
Therefore, by \eqref{eq:pullback}, the total absolute curvature $\tau_S(f)$ is equal to 
the average number of critical points of height functions $h_{\phi} = \tilde{h}_{\phi} \circ f$, i.e.
\begin{equation}\label{eq:average}
    \tau_S (f) = \frac{1}{\operatorname{vol}(\mathcal{S}^{n+r-1}(1))} \int_{S \setminus \nu(C)} \# \operatorname{crit} (h_\phi) \, \omega_S,
\end{equation}
where $\# \operatorname{crit} (h_\phi)$ denotes the number of critical points of $h_\phi$ on $M$.
The total absolute curvature $\tau_S(f)$ does not depend on the choice of a pair $(N,\theta^{\perp})$ \cite[Theorem~A]{Koi2}.
On the other hand, $\tau_S(f)$ depends on the choice of a unit ellipsoid $S$.
However, the condition $\tau_S(f) = 2$ does not depend on the choice of $S$ as shown in the following fact.
\begin{fact}\cite[Theorem 4.1]{Koi2}\label{fa:min}
Let $M$ be an oriented compact $n$-dimensional manifold, 
and let $\gamma(M)$ denote the minimum number of critical points of Morse functions on $M$.
Then, for an equiaffine immersion $f:(M,\nabla,\theta) \to (\R^{n+r},\tilde{\nabla},\omega)$, 
the total absolute curvature $\tau_S(f)$ is equal to $\gamma(M)$ if and only if
$\# \operatorname{crit} (h_\phi ) = \gamma(M)$ holds for every Morse height function $h_\phi$.
In particular, $\tau_S (f) = 2$ if and only if $\# \operatorname{crit} (h_\phi ) = 2$ holds for every Morse height function $h_\phi$.
\end{fact}
From this Fact~\ref{fa:min}, it follows that the condition $\tau_S(f)=2$ does not depend on the choice of the unit ellipsoid $S$.
Therefore, the condition $\tau_S(f)=2$ is a geometric property of the equiaffine immersion $f$ itself, 
and it is natural to expect that this condition characterizes a geometric property of the image $f(M)$.

\begin{remark}
For an equiaffine immersion $f:(M,\nabla,\theta) \to (\R^{n+r},\tilde{\nabla},\omega)$,
if the total absolute curvature $\tau_S(f)$ is equal to $2$ with respect to some unit ellipsoid $S$,
then $\tau_S(f)$ is equal to $2$ with respect to any unit ellipsoid $S$.
\end{remark}

There are many examples of equiaffine immersions with minimal total absolute curvature $2$.
For example, ellipsoids have minimal total absolute curvature $2$ (see Example \ref{ex:ellipsoid}).
More generally,
centro-equiaffine immersions which are defined on the sphere 
also have minimal total absolute curvature (see \cite[Example 1,2]{Koi2}).

\section{Proof of Theorem~\ref{thm:main}}\label{sec:proof}

In this section, we prove Theorem~\ref{thm:main}. 
In Subsection~\ref{subsec:red}, 
we prove Theorem~\ref{thm:codim-one}, 
showing that if the total absolute curvature is $2$, 
then the image is contained in an $(n+1)$-dimensional affine subspace of $\R^{n+r}$.
In other words, the image of an equiaffine immersion with minimal total absolute curvature 
is a hypersurface in an $(n+1)$-dimensional affine space.
By Theorem~\ref{thm:codim-one}, 
to prove the equivalence between 
the minimality of the total absolute curvature and convexity, 
it suffices to consider only the case where the codimension is $1$. 
This will be done in Subsection~\ref{subsec:cvx}. 
We prove that the total absolute curvature is equal to $2$ 
if and only if the image of the equiaffine immersion 
is a convex hypersurface (Theorems~\ref{thm:cvx_to_min} and \ref{thm:min_to_cvx}). 
Finally, combining these results, we complete the proof of Theorem~\ref{thm:main}.

\subsection{Reduction of codimension}\label{subsec:red}

In this subsection,
we prove Theorem~\ref{thm:codim-one} that 
the image of an equiaffine immersion with minimal total absolute curvature 
$2$ is contained in an $(n+1)$-dimensional affine subspace.
To show Theorem~\ref{thm:codim-one}, we first show Lemma~\ref{lem:reduction} that 
the codimension can be reduced by one. 
However, a decrease in the codimension means that the ambient space and its structure change.
Hence, there is a possibility that the equiaffine property may be lost or that the value of the total absolute curvature may change. 
In Proposition~\ref{prop:equiaffine_restrict}, we show that by introducing an appropriate equiaffine structure, 
the total absolute curvature does not change even when 
the codimension is reduced. Finally, using these results, 
we prove Theorem~\ref{thm:codim-one}.

We fix a unit ellipsoid $S$.
Let $f : (M,\nabla,\theta) \to (\R^{n+r},\tilde{\nabla},\omega)$ be an equiaffine immersion, 
and let $(N,\theta^{\perp})$ be a pair of a transversal bundle and a volume element of $f$.
We denote the transversal ellipsoid bundle by $\pi:B \to M$.
Next, we explain the linear correspondence between elements of 
$\wedge^{n+r-1} \R^{n+r}$ and height functions, 
which will be used in the proofs of Lemma~\ref{lem:hyperplane} and Proposition~\ref{prop:equiaffine_restrict}.
For $\phi \in \wedge^{n+r-1} \R^{n+r}$, we write
$
\phi = \vect{v}_1 \wedge \cdots \wedge \vect{v}_{n+r-1}.
$
By the definition of the height function, 
$\tilde{h}_\phi$ can be written as
\[
\tilde{h}_\phi(\vect{v})
=
\omega(\vect{v}_1, \ldots, \vect{v}_{n+r-1}, \vect{v})
\qquad
(\vect{v} \in \R^{n+r}).
\]
Since the height function is independent of the representation of $\phi$, 
we write
\[
\tilde{h}_\phi(\vect{v}) = \omega(\phi, \vect{v}).
\]
Furthermore, for $\phi_1, \phi_2 \in \wedge^{n+r-1} \R^{n+r}$ and $a,b \in \R$, 
we have
\[
\tilde{h}_{a\phi_1 + b\phi_2}(\vect{v})
=
\omega(a\phi_1 + b\phi_2, \vect{v})
=
a\,\omega(\phi_1, \vect{v})
+
b\,\omega(\phi_2, \vect{v})
=
a\,\tilde{h}_{\phi_1}(\vect{v})
+
b\,\tilde{h}_{\phi_2}(\vect{v}).
\]
Hence,
\begin{equation}\label{eq:linear}
\tilde{h}_{a\phi_1 + b\phi_2}
=
a\,\tilde{h}_{\phi_1}
+
b\,\tilde{h}_{\phi_2}.
\end{equation}

The following lemma is used in the proof of Lemma~\ref{lem:reduction}.

\begin{lemma}\label{lem:hyperplane}
For a point $p \in M$, we fix $\eta \in B_p$.
Write  $\eta = \xi_1 \wedge \cdots \wedge \xi_{r-1}$ and take $\xi_r \in \R^{n+r}$
such that $\theta^\perp(\xi_1,\cdots,\xi_{r-1}, \xi_r)$ = 1.
For $t \in [0,2\pi]$, we define
\[
   \xi (t) := \varphi(t) \, ( \cos t\,\xi_{r-1} + \sin t\,\xi_r )
\]
and
\[
\eta_t := 
\xi_1 \wedge \cdots \wedge \xi_{r-2} \wedge \xi (t),
\]
where $\varphi(t)$ is a positive function such that $\eta_t \in B_p$.
We also define the $(n+r-1)$-dimensional subspace $\Pi_t$ by  
$\Pi_t = \operatorname{span} \{df_p (\vect{e}_1), \cdots, df_p(\vect{e}_n), \xi_1, \cdots, \xi_{r-2}, \xi (t)\}$
where $(\vect{e}_1, \cdots, \vect{e}_n)$ is a frame of $TM$.

Then, for each $\vect{v} \in \R^{n+r}$, there exists $t \in [0,2 \pi]$ such that $\vect{v} \in \Pi_t$.
\end{lemma}

\begin{proof}
When $t=0$ or $t = \pi/2$, we obtain
\[
\eta_0 = \xi_1 \wedge \cdots \wedge \xi_{r-2} \wedge \xi (0) = \xi_1 \wedge \cdots \wedge \xi_{r-2} \wedge \xi_{r-1}
\]
and
\[
\eta_{\pi/2} = \xi_1 \wedge \cdots \wedge \xi_{r-2} \wedge \xi (\pi/2) = \varphi(\pi/2)\, \xi_1 \wedge \cdots \wedge \xi_{r-2} \wedge \xi_{r}
\]
respectively.
Therefore, we have
\begin{align*}
\eta_t &= \xi_1 \wedge \cdots \wedge \xi_{r-2} \wedge \xi (t) \\
&= \varphi (t) (\xi_1 \wedge \cdots \wedge \xi_{r-2} \wedge \cos t\,\xi_{r-1} + \xi_1 \wedge \cdots \wedge \xi_{r-2} \wedge \sin t\,\xi_{r}) \\
&= \varphi (t) \left(\cos t\, \eta_0 + \frac{1}{\varphi(\pi/2)} \sin t\, \eta_{\pi/2}\right)
\end{align*}
for every $t \in [0,2\pi]$.
Thus the height function with respect to $\nu(\eta_t)$ can be written as
\[
   \widetilde{h}_{\nu(\eta_t)}
       = \varphi(t)
         \left(\cos t\,\widetilde{h}_{\nu(\eta_0)}
               + \frac{1}{\varphi(\pi/2)} \sin t\,
               \widetilde{h}_{\nu(\eta_{\pi/2})} \right).
\]
Here we use \eqref{eq:linear}.
Since
\[
\widetilde{h}_{\nu(\eta_\pi)} (\vect{v}) = - \varphi(\pi) \widetilde{h}_{\nu(\eta_0)} (\vect{v})
\]
holds when $t = \pi$,
if $\widetilde{h}_{\nu(\eta_0)} (\vect{v})$ does not vanish for a vector $\vect{v} \in \R^{n+r}$,
then $\widetilde{h}_{\nu(\eta_0)} (\vect{v})$ and $\widetilde{h}_{\nu(\eta_\pi)} (\vect{v})$ have 
opposite signs.
Therefore, by the intermediate value theorem, 
there exists $t \in (0,\pi)$ such that 
$\widetilde{h}_{\nu(\eta_t)}(\vect{v})=0$.
The same statement holds in the case of $\widetilde{h}_{\nu(\eta_{\pi/2})} (\vect{v}) \neq 0$.
If both $\widetilde{h}_{\nu(\eta_0)} (\vect{v})$ and $ \widetilde{h}_{\nu(\eta_{\pi/2})} (\vect{v})$ vanish,
then $\widetilde{h}_{\nu(\eta_t)}(\vect{v})=0$ for every $t \in [0,2\pi]$.
As a result, for each $\vect{v} \in \R^{n+r}$, there exists $t \in [0,\pi]$ such that 
$\widetilde{h}_{\nu(\eta_t)}(\vect{v})=0$.
This implies that every $\vect{v} \in \R^{n+r}$ belongs to some $\Pi_t$.
\end{proof}

In the following Lemma~\ref{lem:reduction}, 
we show that the image of an equiaffine immersion in $\R^{n+r}$ 
whose total absolute curvature is equal to $2$ 
is contained in an $(n+r-1)$-dimensional affine subspace.

\begin{lemma}\label{lem:reduction}
Suppose that $r \geq 2$.
Let $M$ be an oriented compact $n$-dimensional manifold and 
let $f : (M,\nabla,\theta) \to (\R^{n+r}, \tilde{\nabla}, \omega)$ be an equiaffine immersion.
If the total absolute curvature of $f$ is equal to $2$,
 then the image $f(M)$ is contained 
in an $(n+r-1)$-dimensional affine subspace of $\R^{n+r}$.
\end{lemma}

\begin{proof}
In this proof, we use $\eta_t$ and $\Pi_t$ defined in Lemma~\ref{lem:hyperplane}.
We assume that 
$f(M)$ does not belong to any $(n+r-1)$-dimensional affine subspace of $\R^{n+r}$, 
leading to a contradiction.

Since $\tau_S(f)=2$, the Lipschitz--Killing curvature $G$
does not vanish everywhere on $B$.  
Hence there exists $\eta\in B$ such that $G(\eta)\neq 0$.
Let $p=\pi(\eta)$, write 
$\eta = \xi_1\wedge\cdots\wedge\xi_{r-1}$ and
take $\xi_r$ such that $\theta^\perp(\xi_1,\cdots,\xi_{r-1}, \xi_r)=1$.
By a translation if necessary, 
we may assume that the point $f(p)$ is at the origin.
We set a function $g : [0,2\pi] \to \R$ by
$
  g(t) := G(\eta_t).
$
Since $g$ is continuous and $g(0)\neq 0$, it follows that
$g(t)$ does not vanish everywhere on $[0,2\pi]$.

To derive a contradiction, we construct height functions having at least three critical points.
Because $f(M)$ is not contained in any affine subspace of dimension $(n+r-1)$,
there exists $p_1\in M$ such that
$
   f(p_1)\notin \cap_{t \in [0,2 \pi]} \Pi_t.
$
By Lemma \ref{lem:hyperplane}, there exists $t_1$ such that $f(p_1)\in \Pi_{t_1}$.
Since $f(M)$ is not contained in $\Pi_{t_1}$, we may choose 
$p_2\in M$ such that $f(p_2)\notin \Pi_{t_1}$.
Again using Lemma \ref{lem:hyperplane}, choose $t_2$ such that 
$f(p_2)\in \Pi_{t_2}$.
We now choose $t_3$ satisfying  
\begin{equation}\label{eq:ineq_three}
   \widetilde{h}_{\nu(\eta_{t_3})}(p_2)
   < \widetilde{h}_{\nu(\eta_{t_3})}(p)
   < \widetilde{h}_{\nu(\eta_{t_3})}(p_1),
\end{equation}
and $g(t_3)\neq 0$.
Such $t_3$ can be chosen by the continuity of $g(t)$ and 
the fact that the height functions vary continuously with respect to $t$.
Since $\eta_{t_3}$ belongs to the fiber $B_p$, 
the point $p$ is a critical point of the height function $h_{\nu(\eta_{t_3})}$ (see \cite[Lemma 3.1]{Koi1}).
Therefore, the height function $h_{\nu(\eta_{t_3})}$ has at least three 
critical points on $M$: a minimum, a maximum, and the critical point at $p$.
By the inverse function theorem, there exists an open neighborhood $W$ of $\eta_{t_3}$,  
and the Gauss map $\nu |_W$ is a diffeomorphism onto its image.
We can choose $W$ so small that \eqref{eq:ineq_three} still holds for each $\eta' \in W$.
Then, the height function $h_{\nu(\eta')}$ has at least three critical points for each $\eta' \in W$. 
Since $\nu(W)$ has positive measure in $S$,
the total absolute curvature $\tau_S(f)$ is greater than $2$ (see \eqref{eq:average}).
This contradicts $\tau_S(f) = 2$.
Therefore $f(M)$ is contained in an $(n+r-1)$-dimensional affine subspace.  
\end{proof}

Next, we show in Proposition~\ref{prop:equiaffine_restrict} that if an equiaffine immersion $f$ 
is contained in an $(n+r-1)$-dimensional affine subspace $L$, 
then one can choose an equiaffine structure of $L$ so that the total absolute curvature 
does not change even when $f$ is regarded as a map into $L$.
The following volume element $\omega_L$ on $L$ 
and the relation \eqref{eq:psi_xi} 
between the height functions $\tilde{h}^L_\psi$ and $\tilde{h}_{\psi \wedge \xi}$ 
are used in the proof of Proposition~\ref{prop:equiaffine_restrict}.
Let $L'$ be the $(n+r-1)$-dimensional linear subspace parallel to $L$.
We take a vector $\xi \in \R^{n+r}\setminus L'$ and
 define a volume element $\omega_L$ on $L'$ by
\begin{equation}\label{eq:omega_L}
   \omega_{L}(X_1, \cdots , X_{n+r-1})
   := \omega( X_1, \cdots , X_{n+r-1}, \xi)
   \quad (X_1, \cdots , X_{n+r-1} \in L').
\end{equation}
Next, we describe the properties of the height function $\tilde{h}^L_\psi$ on $L'$.
For $\psi \in \wedge^{n+r-2} L'$, the height function $\tilde{h}^L_\psi$ on $L'$ is defined by using $\omega_L$ (see \eqref{eq:height-function}).
It can be written as
$\tilde{h}^L_\psi (\vect{v}) = \omega_L (\psi, \vect{v})$.
On the other hand, for $\psi \wedge \xi \in \wedge^{n+r-1} \R^{n+r}$, 
the height function $\tilde{h}_{\psi \wedge \xi}$ can be written as
$\tilde{h}_{\psi \wedge \xi} (\vect{v}) = \omega(\psi \wedge \xi, \vect{v})$.
Hence, if $\psi = \vect{w}_1 \wedge \cdots \wedge \vect{w}_{n+r-2}$, then
\[
  \omega(\psi \wedge \xi, \vect{v}) = \omega(\vect{w}_1, \cdots , \vect{w}_{n+r-2}, \xi, \vect{v})
= - \omega(\vect{w}_1, \cdots , \vect{w}_{n+r-2}, \vect{v}, \xi)
= -\omega_L (\psi, \vect{v})
\]
holds.
Therefore, we have
\begin{equation}\label{eq:psi_xi}
\tilde{h}^L_\psi = - \tilde{h}_{\psi \wedge \xi}.
\end{equation}

This is the key step in the proof, 
since unlike the Euclidean case an affine subspace does not naturally inherit an equiaffine structure.
Thus, even if $f(M) \subset L$, it is not automatic that $f$ which is regarded as a map into $L$, 
remains an equiaffine immersion with the same total absolute curvature. 
The following proposition resolves this issue.

\begin{proposition}\label{prop:equiaffine_restrict}
For an equiaffine immersion
$
f:(M,\nabla,\theta)\to(\R^{n+r},\tilde{\nabla},\omega),
$
suppose that the total absolute curvature $\tau_S (f)$ is $2$ 
and the image $f(M)$ is contained in an $(n+r-1)$-dimensional affine subspace $L$.
Denote by $L'$ the $(n+r-1)$-dimensional linear subspace parallel to $L$.
For the volume element $\omega_L$ on $L'$ defined by \eqref{eq:omega_L},
the following statements hold:
\begin{itemize}
\item[$(1)$] 
The restriction of $f$ to $L$,
\[
f_L:(M,\nabla,\theta)\to (L,\tilde{\nabla}|_L,\omega_L).
\]
is an equiaffine immersion.

\item[$(2)$]
For an arbitrary unit ellipsoid $S_L$ in $\wedge^{n+r-2}L'$,
the total absolute curvature $\tau_{S_L}(f_L)$ of $f_L$ with respect to $S_L$ is equal to $2$.
\end{itemize}

\end{proposition}

\begin{proof}
First, we prove $(1)$.
It suffices to show that there exists 
a transversal bundle with its volume element $(N_L,\theta_L^{\perp})$ of $f_L$ inducing $\nabla$ and $\theta$.
Since $f$ is equiaffine, 
there exist a transversal bundle and 
its volume element $(N,\theta^{\perp})$ for $f$ that 
induce $\nabla$ and $\theta$.
For each point $p\in M$, define
$(N_L)_p := N_p \cap L'$.
Then $N_L$ is a transversal bundle of $f_L$.
Since the image of $f$ is contained in $L$, 
both $\tilde{\nabla}_X df(Y)$ and $df(\nabla_X Y)$ lie in $L'$.
Hence their difference, 
the affine fundamental form $\alpha$ of $f$, 
takes values in $N_L$.
This shows that $\alpha$ is also the affine fundamental form of $f_L$.
Therefore, the decomposition of $\tilde{\nabla}_X df(Y)$ into its tangential component 
and $N_L$-component for $f_L$ coincides with 
the decomposition for $f$ with respect to $N$.
Thus, the induced connection from $N_L$ coincides with $\nabla$.
Next, we define a volume element $\theta_L^{\perp}$ on $N_L$ by
\[
\theta_L^{\perp}(\xi_1,\cdots,\xi_{r-1})
:=
\theta^{\perp}(\xi_1,\cdots,\xi_{r-1},\xi)
\qquad
(\xi_1,\cdots,\xi_{r-1}\in N_L).
\]
Let $\theta_L$ be the volume element on $M$ induced by $\theta_L^{\perp}$.
Then it can be written as
\[
\theta_L(X_1,\cdots,X_n)
=
\frac{
\omega_L\bigl(df_L(X_1),\cdots,df_L(X_n),\xi_1,\cdots,\xi_{r-1}\bigr)
}{
\theta_L^{\perp}(\xi_1,\cdots,\xi_{r-1})
}.
\]
where $X_1,\cdots,X_n$ are tangent vectors on $M$.
By the definitions of $\omega_L$ and $\theta_L^{\perp}$, the induced volume element $\theta_L$ coincides with $\theta$.
Hence $(\nabla,\theta)$ is induced by $(N_L,\theta_L^{\perp})$, and therefore
$
f_L:(M,\nabla,\theta)\to(L,\tilde{\nabla}|_L,\omega_L)
$
is an equiaffine immersion.

Next, we prove $(2)$.
We have the decomposition
\[
\wedge^{n+r-1}\R^{n+r}
=
\bigl(\wedge^{n+r-2}L'\bigr)\wedge\xi
\oplus
\wedge^{n+r-1}L'.
\]
Hence, for every $\psi \in S_L$, 
there exist $a(\neq 0),b \in R$ and $\zeta \in \wedge^{n+r-1}L'$ satisfying
\[
a\,\psi \wedge \xi+b\,\zeta \in S.
\]
Denote $\phi := a\,\psi \wedge \xi+b\,\zeta$.
Then $\tilde{h}_\phi = a\,\tilde{h}_{\psi \wedge \xi} + b\,\tilde{h}_\zeta$ holds by \eqref{eq:linear}.
The height function $h_\phi=\tilde h_\phi\circ f$ can be written as
$
h_\phi=a\,h_{\psi \wedge \xi}+b\,h_\zeta.
$
Here, $f(M)$ is contained in $L$, and $L'$ is parallel to $L$.
Therefore, we obtain that the function $h_\zeta=\tilde h_\zeta\circ f$ is constant.
Hence $h_\phi = a\,h_{\psi \wedge \xi} + c$ holds where $c$ is a constant.
As a result, we have
\[
h_\phi = - a\,h^L_{\psi} + c
\]
by \eqref{eq:psi_xi}.
Therefore a point $p\in M$ is a critical point of $h_\phi$ if and only if 
it is a critical point of $h^L_\psi$, and we have
\[
\#\operatorname{crit}(h_\phi)=\#\operatorname{crit}(h^L_\psi)
\]
when $h_\phi$ is a Morse function.
Since the total absolute curvature $\tau_S(f)$ is equal to $2$, 
the Morse height function $h_\phi$ has exactly two critical points by Fact~\ref{fa:min}.
Hence, the Morse height function $h^L_\psi$ has also exactly two critical points.
This implies that the total absolute curvature $\tau_{S_L}(f_L)$ is equal to $2$ 
for any unit ellipsoid $S_L$ in $\wedge^{n+r-2}L'$.

\end{proof}

By this Proposition~\ref{prop:equiaffine_restrict},
 it becomes possible to decrease the codimension while preserving the minimality of the total absolute curvature of $f$. 
This allows us to apply Lemma~\ref{lem:reduction} repeatedly in the following Theorem~\ref{thm:codim-one}.
We prove that an equiaffine immersion with minimal total absolute curvature 
is contained in an $(n+1)$-dimensional affine subspace.

\begin{theorem}\label{thm:codim-one}
Let $(M,\nabla,\theta)$ be an oriented compact $n$-dimensional manifold $M$ 
with an equiaffine structure $(\nabla, \theta)$, and 
$f:(M,\nabla,\theta)\to(\R^{n+r},\tilde{\nabla},\omega)$ an equiaffine immersion.
If the total absolute curvature $\tau_S(f)$ of $f$ is equal to $2$,
 then the image $f(M)$ is contained in an $(n+1)$-dimensional affine subspace of $\R^{n+r}$.
\end{theorem}

\begin{proof}
By Lemma~\ref{lem:reduction}, the image of $f$ is contained in an $(n+r-1)$-dimensional affine subspace.
Let $L$ denote this affine subspace.
Then, by $(1)$ of Proposition~\ref{prop:equiaffine_restrict}, 
the restriction of $f$ to $L$ induces an equiaffine immersion
$f_L:(M,\nabla,\theta)\to(L,\tilde{\nabla}|_L,\omega_L)$.
By $(2)$ of Proposition~\ref{prop:equiaffine_restrict}, 
$\tau_S(f)=\tau_{S_L}(f_L)=2$ holds for an arbitrary unit ellipsoid $S_L$ in $\wedge^{n+r-2}L'$.
Hence the property that the total absolute curvature equals $2$ is preserved under reduction of codimension.
Repeating this process until the codimension becomes $1$, we obtain the desired conclusion.
\end{proof}

\subsection{Minimal total absolute curvature and convexity}\label{subsec:cvx}

In this subsection, we prove Theorem~\ref{thm:main}.
We show that an equiaffine immersion is a convex hypersurface 
if and only if the total absolute curvature is equal to $2$ in the codimension-one case.
We prove that an equiaffine immersion which is a convex hypersurface has minimal total absolute curvature
in Theorem~\ref{thm:cvx_to_min}
and prove the converse in Theorem~\ref{thm:min_to_cvx}.  
Finally, combining Theorems~\ref{thm:cvx_to_min}, \ref{thm:min_to_cvx} and the results of subsection~\ref{subsec:red}, 
we complete the proof of Theorem~\ref{thm:main}.

By Theorem \ref{thm:codim-one}, we can regard an equiaffine immersion $f:M \to \R^{n+r}$ with $\tau_S(f) =2$ 
as the immersion into $(n+1)$-dimensional affine space $\R^{n+1}$ without loss of generality.
First, we prove that a convex hypersurface has minimal total absolute curvature.

\begin{theorem}\label{thm:cvx_to_min}
Let $(M,\nabla,\theta)$ be an oriented compact $n$-dimensional manifold $M$ 
with an equiaffine structure $(\nabla, \theta)$ and 
$f : (M,\nabla,\theta) \to (\R^{n+1}, \tilde{\nabla}, \omega)$ be an equiaffine immersion.
If an equiaffine immersion $f$ is a convex hypersurface,
then the total absolute curvature $\tau_S(f)$ of $f$ is equal to $2$.
\end{theorem}

\begin{proof}
Denote by $C$ the set of critical points of the Gauss map $\nu$.
We prove the statement by contradiction.  
Assume that there exists a $\phi \in S \setminus \nu(C)$ such that the height function
$h_{\phi}$ has three critical points $p_1, p_2, p_3 \in M$.
Since $\phi$ is not a critical value of $\nu$, the height function $h_{\phi}$ is a Morse function.
We denote by $\Pi_{\phi}$ the hyperplane in $\R^{n+r}$ corresponding to $\phi$.
For each $p_i$, since $p_i$ is a critical point of the height function $h_{\phi}$, 
the hyperplane $\mathcal{P}_i$ parallel to $\Pi_{\phi}$ passing through $f(p_i)$ 
must be a supporting hyperplane of the convex hypersurface $f(M)$.
Hence, among the three supporting hyperplanes, at least two hyperplanes must coincide.
Therefore, by relabeling if necessary, we may assume $\mathcal{P}_1 = \mathcal{P}_2$.
We define
\[
   \ell(t) := t f(p_1) + (1-t)f(p_2) \quad ( t \in [0,1]).
\]
Because $f$ is convex, we have
$
   \ell(t) \in H(f(M)) 
$
for every $t \in [0,1]$,
where $H(f(M))$ denotes the smallest convex set containing $f(M)$.
Since $\ell(t) \in \mathcal{P}_1(=\mathcal{P}_2)$, each $\ell(t)$ lies on the supporting hyperplane.
Hence
$
   \widetilde{h}_{\phi}(\ell(t))
$
is constant.
On the other hand, because $f$ is convex, the boundary of $H(f(M))$
coincides with $f(M)$ itself.
Thus,
$
   \ell(t) \in f(M)
$
holds for every $t \in [0,1]$.
This implies that $h_{\phi}$ has a one-parameter family of critical points,
which shows that $h_{\phi}$ is not a Morse function.
This contradicts the assumption $\phi \in S \setminus \nu(C)$.
\end{proof}

Next, we prove that an equiaffine immersion with minimal total absolute curvature 
is a convex hypersurface in Theorem~\ref{thm:min_to_cvx}.
When the codimension is $1$, the transversal ellipsoid bundle $B$ is written as
$B = \{(p,\pm \mu(p)) \mid p \in M\} \subset M \times \R$, 
where $\mu$ is a positive function on $M$ such that $\tilde{\nu}_N (p,\pm \mu(p)) \in S$.
In other words, 
\begin{equation}\label{eq:nu_mu}
   \tilde{\nu}_N (p,\pm \mu(p)) = \pm \mu(p) \, df_p (\vect{e}_1) \wedge \cdots \wedge df_p (\vect{e}_n) \in S
\end{equation}
holds, where $\{\vect{e}_1,\cdots,\vect{e}_n\}$ is a frame of $T_p M$ such that $\theta (\vect{e}_1 , \cdots \vect{e}_n) = 1$.
Thus, we regard the Gauss map $\nu$ as a map on $M$ by setting
\[
   \nu(p) := \tilde{\nu}_N (p, \mu(p)) \quad (p \in M).
\]
Then, by \eqref{eq:nu_mu}, 
the value of the Gauss map $\nu(p)$ corresponds to the tangent hyperplane of $f(M)$ at $f(p)$.

The following lemma is used in the proof of Theorem~\ref{thm:min_to_cvx}.

\begin{lemma}\label{lem:nu_critical_point}
Denote by $C$ the set of critical points of the Gauss map $\nu:M \to S$
and by $\mathcal{P}_p$ the tangent hyperplane of $f(M)$ at $f(p)$.
We assume that the total absolute curvature $\tau_S(f)$ of $f$ is equal to $2$.
\begin{itemize}
   \item[$(1)$] If a point $q$ belongs to the boundary $\partial C$ of $C$, 
   then the tangent hyperplane $\mathcal{P}_q$ is a supporting hyperplane of $f(M)$.

   \item[$(2)$] If a point $q$ belongs to the interior $C^{\circ}$ of $C$, 
then there exists a point $q_0\in \partial C$ such that
$\mathcal{P}_q = \mathcal{P}_{q_0}$ and $\mathcal{P}_q$ is a supporting hyperplane of $f(M)$.

\end{itemize}
\end{lemma}

\begin{proof}
First, we prove $(1)$.
Since $q$ is a boundary point of $C$, 
there exists a sequence $\{p_n\}_{n \in \boldsymbol{N}}$ such that 
\[
\lim_{n \to \infty} p_n = q \mbox{ and } p_n \in M \setminus C \quad (n \in \boldsymbol{N}).
\]
For each $p_n$, since $h_{\nu(p_n)}$ is a Morse function and 
$\tau_S(f)=2$, the number of critical points is $2$.
Hence the tangent hyperplane at each critical point is a supporting hyperplane.
If it is not a supporting hyperplane, then the critical point $p_n$ would not be 
a minimum or maximum, and therefore $h_{\nu(p_n)}$ would have at least 
three critical points, which is a contradiction.
Thus, for each $p_n$, the tangent hyperplane $\mathcal{P}_{p_n}$ is a supporting hyperplane.
Without loss of generality, we may assume that $p_n$ attains minimum.
Then, we have
\[
   h_{\nu(p_n)}(x) \ge h_{\nu(p_n)}(p_n) \quad (\forall x \in M).
\]
Fix $x \in M$ and take the limit as $n\to\infty$.  
Then
\[
   h_{\nu(q)}(x) \ge h_{\nu(q)}(q) 
\]
holds. 
Therefore, $h_{\nu(q)}(x) \ge h_{\nu(q)}(q)$ holds for every $x \in M$.
Hence the tangent hyperplane $\mathcal{P}_q$ at $f(q)$ is a supporting hyperplane (see Figure~\ref{fig:tangent}).

Next, we prove $(2)$.
Since $\nu(C)$ is a measure zero subset of $S$ by Sard's theorem,
every neighborhood of $\nu(q)$ contains regular values of $\nu$.
Thus there exists a sequence $\{\phi_n\}$ in $S \setminus \nu(C)$
such that $\phi_n \to \nu(q)$ as $n\to\infty$.
Since each $\phi_n$ is a regular value, there exists a regular sequence $\{p_n\}_{n \in \boldsymbol{N}}$
such that $\nu(p_n) = \phi_n \, (n \in \boldsymbol{N})$.
Because $M$ is sequentially compact, 
the sequence $\{p_n\}_{n \in \boldsymbol{N}}$ has a convergent subsequence.
Let $\{p_{n_k}\}_{k \in \boldsymbol{N}}$ be such a subsequence and let
$
    q_0 := \lim_{k\to\infty} p_{n_k}.
$
Then
\[
  \nu(q_0)
   = \nu\!\left(\lim_{k\to\infty} p_{n_k}\right)
   = \lim_{k\to\infty} \nu(p_{n_k})
   = \lim_{k\to\infty} \phi_{n_k}
   = \nu(q)
\]
holds. Hence $\nu(q_0) = \nu(q)$.
The tangent hyperplane of $f(M)$ at $f(q)$ 
is corresponding to the value of the Gauss map $\nu(q)$.
Therefore, $\mathcal{P}_q = \mathcal{P}_{q_0}$.
In addition, $q_0$ is a critical point.
Since every neighborhood of $q_0$ contains points of $M\setminus C$
by the construction of the sequence, $q_0$ lies on the boundary of $C$ (see Figure~\ref{fig:nu_c_p}).
By the argument of $(1)$ of this lemma, $\mathcal{P}_q = \mathcal{P}_{q_0}$ is also a supporting hyperplane.
\end{proof}

\begin{figure}
\begin{center}
\includegraphics[width=0.75\textwidth]{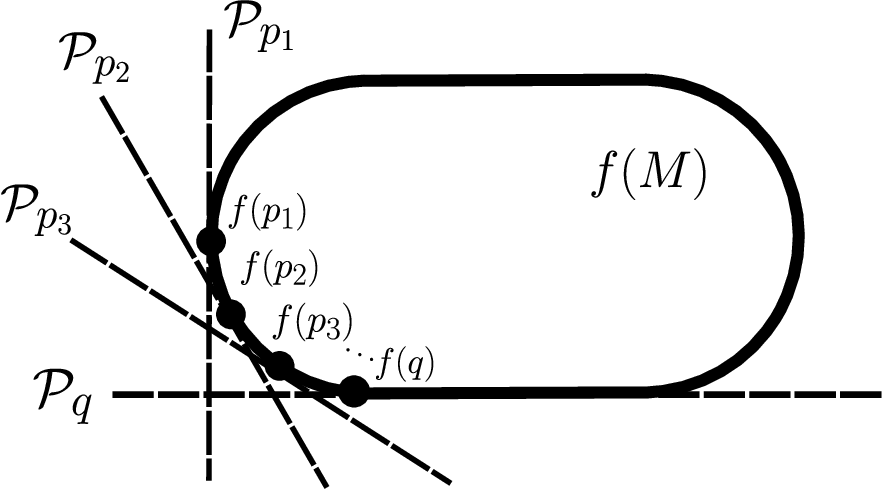}
\end{center}
\caption{Tangent hyperplanes at $f(p_n)$ and $f(q)$  (the proof of $(1)$ of Lemma~\ref{lem:nu_critical_point}).}
\label{fig:tangent}
\end{figure}

\begin{figure}
\begin{center}
\includegraphics[width=0.75\textwidth]{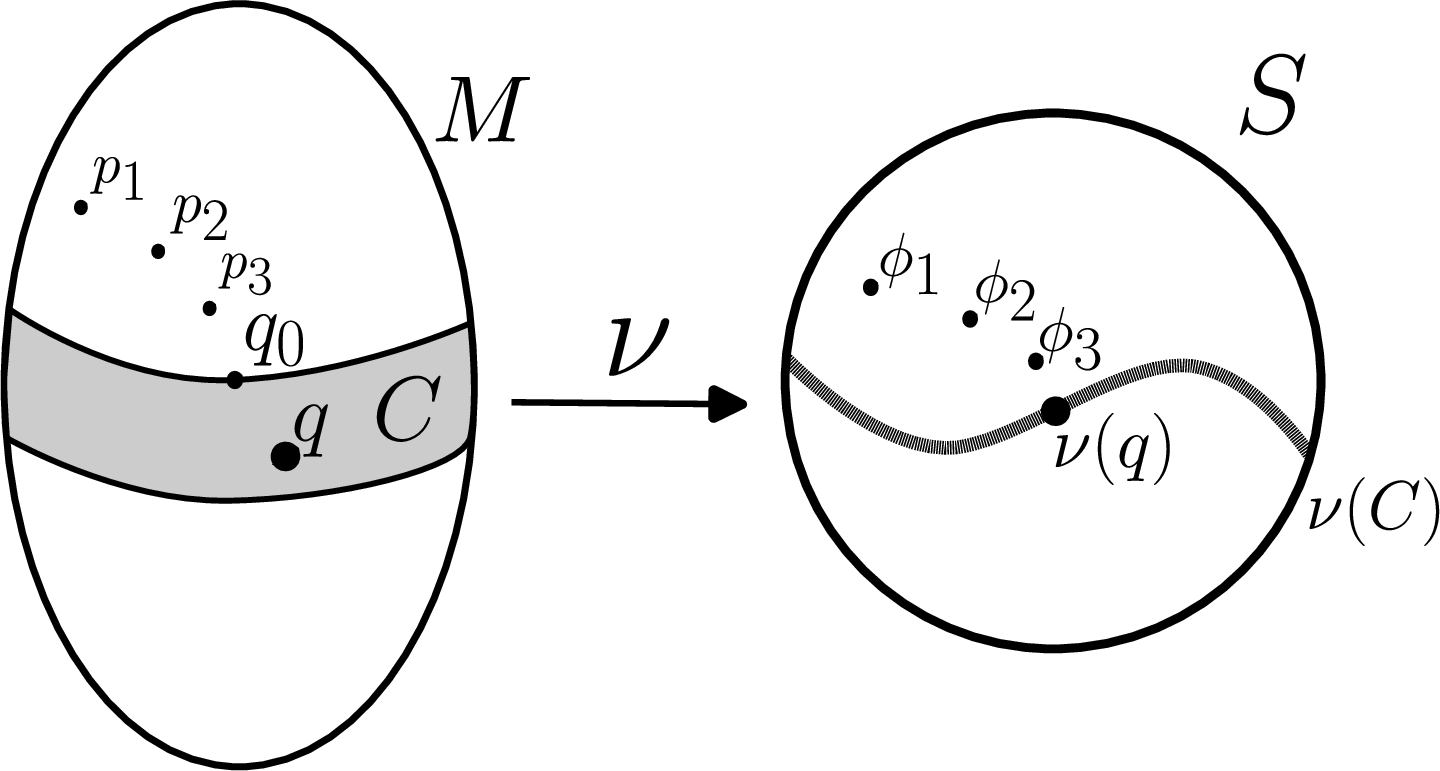}
\end{center}
\caption{ The correspondence between sequences $\{\phi_n\}$ and $\{p_n\}$  (the proof of $(2)$ of Lemma~\ref{lem:nu_critical_point}).}
\label{fig:nu_c_p}
\end{figure}

By using Lemma~\ref{lem:nu_critical_point}, 
we show that the image of an equiaffine immersion with minimal total absolute curvature is a convex hypersurface in the following Theorem~\ref{thm:min_to_cvx}.

\begin{theorem}\label{thm:min_to_cvx}
Let $(M,\nabla,\theta)$ be an oriented compact $n$-dimensional manifold with an equiaffine structure $(\nabla, \theta)$ and 
$f : (M,\nabla,\theta) \to (\R^{n+1}, \tilde{\nabla}, \omega)$ be an equiaffine immersion.
If the total absolute curvature $\tau_S(f)$ of $f$ is equal to $2$, 
then $f(M)$ is a convex hypersurface.
\end{theorem}

\begin{proof}
It suffices to show that for every $p\in M$, the tangent hyperplane $\mathcal{P}_p$ 
is a supporting hyperplane of $f(M)$.
If $p\in M\setminus C$, then $h_{\nu(p)}$ is a Morse function and 
$p$ is one of its critical points which attains minimum or maximum.  
Thus the tangent hyperplane at $f(p)$ is a supporting hyperplane.
If $p\in C$, then the statement follows from Lemma \ref{lem:nu_critical_point}.
Therefore for every $p\in M$, the tangent hyperplane at $f(p)$
is a supporting hyperplane of $f(M)$.
Thus $f(M)$ is convex.
\end{proof}

We remark that 
the argument in the proof of Theorem~\ref{thm:min_to_cvx} yields an alternative proof of the classical theorem 
of Chern and Lashof for hypersurfaces when the ambient affine space is equipped with a Euclidean structure. 
The point is that the proof does not use metric notions, 
but only the behavior of height functions and the Gauss map.

Finally, we show the proof of Theorem~\ref{thm:main}.

\begin{proof}[The proof of Theorem~\ref{thm:main}]
If the total absolute curvature $\tau_S (f)$ is equal to $2$, 
the image $f(M)$ of $f$ is contained in an $(n+1)$-dimensional affine subspace by Theorem~\ref{thm:codim-one}.
Denote by $L$ this $(n+1)$-dimensional affine subspace.
Because of Proposition~\ref{prop:equiaffine_restrict}, 
the total absolute curvature is still $2$ 
when $f$ is regarded as an equiaffine immersion into $L$.
Hence, by Theorem~\ref{thm:min_to_cvx}, 
the image $f(M)$ of $f$ is a convex hypersurface embedded in $L$.
Next, we prove the converse.
If an equiaffine immersion 
$f : (M,\nabla,\theta) \to (\R^{n+r}, \tilde{\nabla}, \omega)$
is a convex hypersurface embedded in an $(n+1)$-dimensional affine subspace $L$, 
then $f$ can be regarded as an equiaffine immersion $f_L$ into $L$.
Moreover, the total absolute curvature of $f_L$ is equal to $2$ by Theorem~\ref{thm:cvx_to_min}.
By Proposition~\ref{prop:equiaffine_restrict}, 
the total absolute curvature does not change even 
after the codimension is reduced.
Hence the total absolute curvature of the original equiaffine immersion $f$ is $2$.
\end{proof}

\section{Examples}\label{sec:example}

In this section, we present two examples of equiaffine immersions whose total absolute
curvature attains the minimal value.
The following first example is an ellipsoid, which is locally strictly convex and whose
affine fundamental form is non-degenerate everywhere.

\begin{example}\label{ex:ellipsoid}
 Let $\R^{n+1}$ be the $(n+1)$-dimensional affine space with the equiaffine structure
 $\tilde{\nabla}$ and 
$\omega.$
Let $(\vect{e}_1 , \cdots , \vect{e}_{n+1})$ be a basis satisfying 
$\omega(\vect{e}_1 , \cdots , \vect{e}_{n+1}) = 1$.
We fix the unit ellipsoid $S$ by
\[
S =
\left\{
\sum_{i=1}^{n+1} a_i\, \vect{e}_1 \wedge \cdots \wedge \widehat{\vect{e}}_i \wedge \cdots \wedge \vect{e}_{n+1}
\ \middle|\ 
\sum_{i=1}^{n+1} a_i^2 = 1
\right\}.
\]
Let $\mathcal{S}^n$ be the unit sphere in the Euclidean space $\E^{n+1}$
 and $x_i : \mathcal{S}^n \to \R$ be the $i$-th coordinate function of $\mathcal{S}^n$ for $i=1,\cdots,n+1$.
Define a map $f : \mathcal{S}^n \to \R^{n+1}$ by
\[
 f(p) := x_1 (p) \vect{e}_1 + \cdots + x_{n+1} (p) \vect{e}_{n+1} \quad (p \in \mathcal{S}^n).
\]
Then the position vector $f(p)$ is an equiaffine transversal vector field.
For $\phi = \sum_{i=1}^{n+1} a_i\, \vect{e}_1 \wedge \cdots \wedge \widehat{\vect{e}}_i \wedge \cdots \wedge \vect{e}_{n+1}$,
a point $p \in \mathcal{S}^n $ is a critical point of the height function $h_{\phi} = \tilde{h}_{\phi} \circ f$
if and only if $df_p (X) \wedge \phi = 0$ for every $X \in T_p \mathcal{S}^n$.
In other words, 
\begin{equation}\label{eq:crit}
    a_1 \cdot (d x_1)_p (X) + \cdots + a_{n+1} \cdot (d x_{n+1})_p (X) = 0
\end{equation}
holds for every $X \in T_p \mathcal{S}^n$.
The equation \eqref{eq:crit} means that the differential of the linear functional $\sum a_i x_i$ vanishes on $T_p \mathcal{S}^n$.
Hence $p$ is parallel to the vector $(a_1,\cdots,a_{n+1})$, which implies that critical points are $p = \pm (a_1,\cdots,a_{n+1})$.
Therefore, for every $\phi \in S$,
the height function $h_{\phi}$ has exactly two critical points.
Thus, the total absolute curvature $\tau_S(f)$ is equal to $2$.
\end{example}

On the other hand, 
the following second example is a convex surface that admits degenerate points where
the affine fundamental form degenerates.
However, even in such a situation, we show that the surface has a
degenerate metric with good geometric properties.
The following example shows that even when the affine fundamental form degenerates, 
convex hypersurfaces with minimal total absolute curvature still admit meaningful geometric structures. 
It also suggests new directions for affine differential geometry 
beyond the non-degenerate affine metric.

\begin{example}\label{ex:semidefinite}
Let $\Sigma$ be a closed surface in the affine space $\R^{3}$ defined by
\[
   \Sigma := \{ (x,y,z) \in \R^3 \mid (z-\sqrt{x^2 + y^2})^4 + (z+\sqrt{x^2 + y^2})^4 = 16 \}.
\]
Then $\Sigma$ is a convex surface in the affine space (see Figure~\ref{fig:semi}).
Define functions $E(u)$ and $F(u)$ by
\[
   E(u) := u-(1-u^4)^{1/4},
   \qquad
   F(u) := u+(1-u^4)^{1/4}.
\]
Local coordinate representations of $\Sigma$ are given by
$f_{\pm}: [-(\frac{1}{2})^{1/4},(\frac{1}{2})^{1/4}] \times [-\pi,\pi] \to \R^3$,
\[
   f_{\pm}(u,v)
   :=
   \left(
     E(u) \cos v,
     E(u) \sin v,
     \pm F(u)
   \right).
\]
From now on, we consider $f_{+}$.
Set
$
   \delta(u) := \sqrt{(E'(u))^2 + (F'(u))^2},
$
where $'$ denotes differentiation with respect to $u$.
Define a vector field $\xi$ by
\[
   \xi(u,v)
   :=
   \frac{1}{\delta(u)}
   \bigl(
     F'(u)\cos v,\,
     F'(u)\sin v,\,
     -E'(u)
   \bigr).
\]
Then $\xi$ is an equiaffine transversal vector field.
The determinant of the affine fundamental form $\alpha_\xi$ is given by (cf.\ \eqref{eq:affine-fund})
$
   \det \alpha_\xi
   =
   u^2 \beta(u)
$
where 
\begin{align*}
\beta (u) = \frac{6}{\delta(u)^2(1-u^4)^{13/4}}
   \cdot
   \bigl(-u+(1-&u^4)^{1/4}\bigr)
   \bigl(-1+u^4+u(1-u^4)^{3/4}\bigr) \\
   &\cdot
   \bigl(-1+u^3(1-u^4)^{1/4}+u(1-u^4)^{3/4}\bigr).
\end{align*}
Then, $\beta(u) > 0$.
Let ${\rm Deg}(\alpha_\xi)$ denote the set of points at which $\alpha_\xi$ degenerates.
Then ${\rm Deg}(\alpha_\xi)$ is given by
$
 {\rm Deg} (\alpha_\xi) = \{ (0,v) \mid v \in [-\pi,\pi] \}.
$
We set a $C^{\infty}$ function
$\lambda(u) := u \sqrt{ \beta(u) }$.
Thus $\lambda (u) = 0$ if and only if $u=0$.
In addition, 
$\lambda'(0)$ does not vanish. 
Then $\alpha_\xi$ satisfies the conditions
of a \emph{Kossowski metric} (see \cite{HHNSUY}).
This is a positive semi-definite metric describing the properties of surfaces with singular points.

\end{example}

\begin{figure}[htb]
\centering
 \begin{tabular}{c}
\resizebox{6cm}{!}{\includegraphics{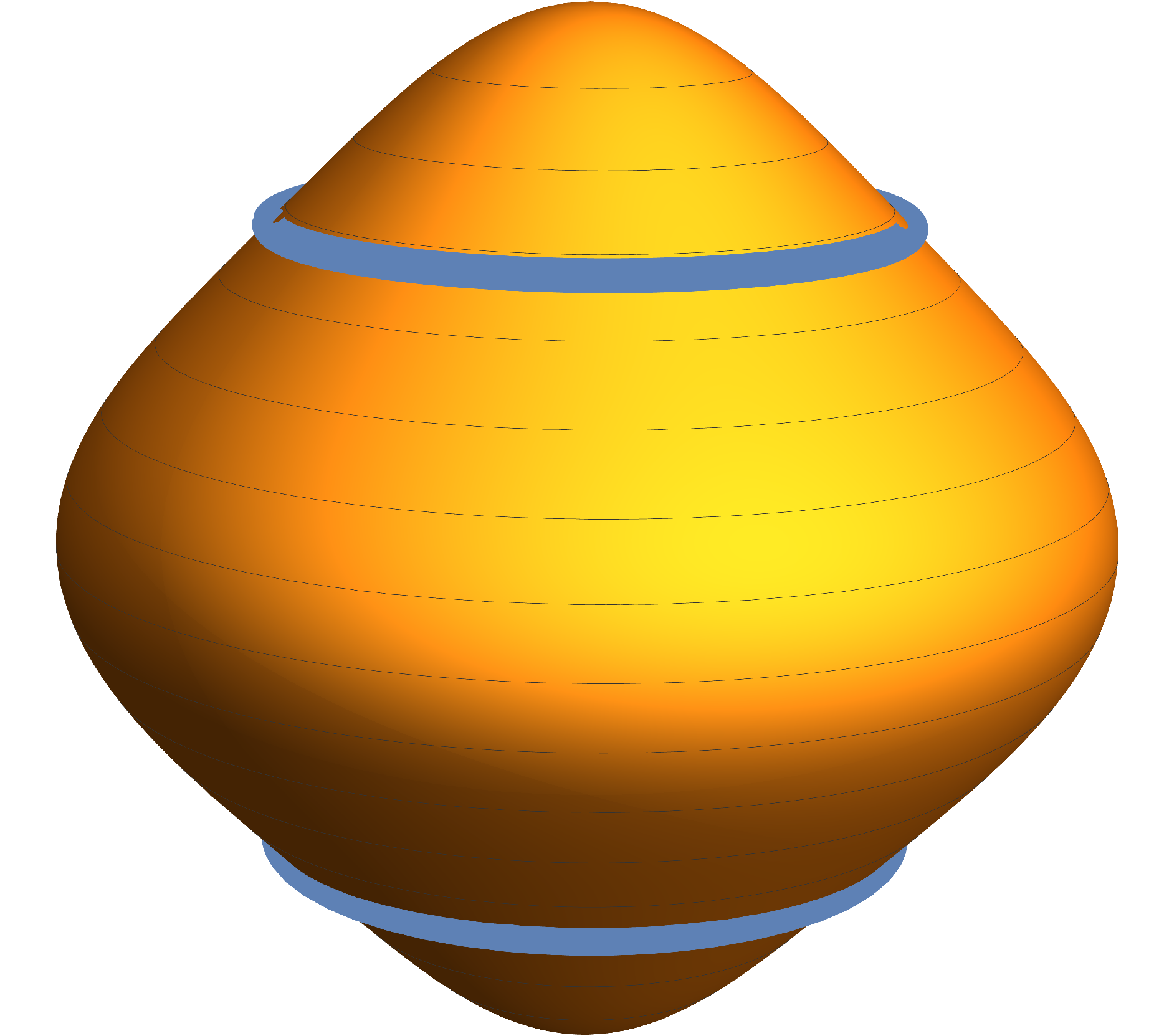}}
\end{tabular}  
\caption{The convex surface $\Sigma$ in Example \ref{ex:semidefinite} 
and the image of degenerate locus ${\rm Deg} (\alpha_\xi)$ (blue curves).
   The affine fundamental form of $\Sigma$ is positive semi-definite}
\label{fig:semi}
\end{figure}

\begin{acknowledgements}
The author thanks Professors Atsufumi Honda, Naoyuki Koike and Hiroyuki Tasaki 
for fruitful discussions and valuable comments.
This work was supported in part by JST SPRING, Grant Number JPMJSP2178.
\end{acknowledgements}

%
%
%
%

\end{document}